\documentclass[11pt]{amsart}

\usepackage{amsmath,amssymb,amsthm,amscd,esint,color}
\usepackage{mathrsfs}
\usepackage{graphicx}
\usepackage{array}
\usepackage{xcolor}
\usepackage{comment}
\usepackage{esint}
\usepackage{citeref}

\numberwithin{equation}{section}
\allowdisplaybreaks

\newtheorem{theorem}{Theorem}[section]
\newtheorem{proposition}[theorem]{Proposition}
\newtheorem{lemma}[theorem]{Lemma}
\newtheorem{corollary}[theorem]{Corollary}

\theoremstyle{definition}
\newtheorem{definition}[theorem]{Definition}
\newtheorem{example}[theorem]{Example}
\newtheorem{remark}[theorem]{Remark}

\begin{document}

\title{Compact sets in Banach lattices over spaces of homogeneous type}
\author{Ay\c{s}enur Aydo\u{g}du}
\address{Ay\c{s}enur Aydo\u{g}du\\
Department of Mathematics \\
Ankara University\\
Ankara, Turkey}
\email{anaydogdu@ankara.edu.tr} 
\author{Amiran Gogatishvili}
\address{Amiran Gogatishvili\\
Institute of Mathematics of the 
Academy of Sciences of the Czech Republic, 
\'Zitna 25 \\
115 67 Prague 1, Czech Republic}
\email{gogatish@math.cas.cz}
\author{ Yoshihiro Sawano}
\address{Yoshihiro Sawano  \\Department of Mathematics,
Graduate School of Science and Engineering,
Chuo University,
1-13-27 Kasuga,
Bunkyo-ku,
Tokyo 112-8551, Japan}
\email{yoshihiro-sawano@celery.ocn.ne.jp} 

\author{Daiki Takesako}
\address{Daiki Takesako\\ Department of Mathematics,
Graduate School of Science and Engineering,
Chuo University,
1-13-27 Kasuga,
Bunkyo-ku,
Tokyo 112-8551, Japan}
\email{takesako.math@gmail.com}

\begin{abstract}
We prove a Kolmogorov--Riesz compactness theorem for Banach lattices over spaces of homogeneous type.
Especially, we consider the one for $L^1$.
Our characterization is formulated in terms of boundedness, tightness, and approximation by averaging operators, which replace translations in the absence of a group structure. The result requires neither the Vitali covering property nor continuity assumptions on the measure of balls, lower volume bounds, or finiteness of the underlying measure space, thereby extending several recent compactness criteria in the literature.

Our approach begins with a new proof of the Kolmogorov--Riesz theorem in $L^1$, which avoids reflexivity and yields a unified treatment of $L^p$-spaces for $1\le p<\infty$. We then establish an approximation theorem based on Christ's dyadic cubes and use it to characterize the closure of compactly supported continuous functions in general ball Banach function spaces. Applications include compactness criteria for rearrangement-invariant Banach function spaces and Morrey spaces. In the Euclidean setting, we also identify the resulting approximation space with the classical heat-semigroup (equivalently, mollifier) closure.
\end{abstract}

\date{August 2026}

\thanks{
 The research of  A. Gogatishvili was partially supported by RVO: 67985840, Institute of Mathematics of the Czech Academy of Sciences. Part of the work on this project was carried out during Y.~Sawano's visit and A.~Aydoğdu's Erasmus visit to the Institute of Mathematics.
}

\maketitle

\section{Introduction}

Compactness criteria play a fundamental role in functional analysis and its applications to partial differential equations, harmonic analysis, and the calculus of variations. Among them, the classical Kolmogorov--Riesz theorem provides a characterization of relatively compact subsets of $L^p(\mathbb{R}^n)$ in terms of boundedness, tightness, and equicontinuity under translations. Since its appearance, this theorem has been extended in many directions, including weighted spaces, variable exponent spaces, Morrey spaces, and other Banach function spaces.

On general metric measure spaces, however, translations are no longer available. Consequently, one cannot formulate compactness in terms of translation continuity, and a different mechanism is required to measure local oscillations. Spaces of homogeneous type, introduced by Coifman and Weiss, provide a natural framework for extending harmonic analysis beyond the Euclidean setting. They include Euclidean spaces, manifolds with doubling measures, and many fractal-type metric spaces, while lacking any underlying group structure.

A natural substitute for translations is provided by averaging operators over metric balls. These operators have recently been employed to characterize compact subsets of function spaces on spaces of homogeneous type and on more general metric measure spaces. This approach has the advantage that it depends only on the metric and the measure, and therefore applies in situations where no translation operator is available.

The purpose of this paper is to establish a Kolmogorov--Riesz type compactness theorem for general Banach lattices over spaces of homogeneous type. Our characterization is formulated in terms of boundedness, tightness, and approximation by averaging operators. In contrast to several previous works, our arguments require neither the Vitali covering property nor continuity assumptions on the measure of balls, lower bounds on the measure of balls, or finiteness of the underlying measure space.

Our approach begins with a new proof of the Kolmogorov--Riesz theorem in $L^1$, which avoids reflexivity and provides a unified treatment of $L^p$-spaces for $1\le p<\infty$. We then extend the result to general ball Banach function spaces. As a further application, we establish an approximation theorem based on Christ's dyadic cubes, characterize the closure of compactly supported continuous functions, and derive compactness criteria for rearrangement-invariant Banach function spaces and Morrey spaces.

We adopt the notion by Coifman and Weiss
\cite{CW}.
\begin{definition}[Space of homogeneous type]
Let $X$ be a nonempty set, let $d$ be a quasi-metric on $X$, and let $\mu$ be a nonnegative Borel measure on $X$.

We say that the triple $(X,d,\mu)$ is a \emph{space of homogeneous type} if the following conditions are satisfied.

\begin{enumerate}
\item[(i)] $d:X\times X\to[0,\infty)$ is a quasi-metric; namely,
\begin{align*}
&d(x,y)=0 \iff x=y,\\
&d(x,y)=d(y,x),\\
&d(x,z)\le A_0\bigl(d(x,y)+d(y,z)\bigr)
\end{align*}
for all $x,y,z\in X$, where $A_0\ge1$ is a constant.

\item[(ii)] The measure $\mu$ satisfies the doubling condition: there exists a constant
$C_\mu\ge1$ such that
\[
0<\mu\bigl(B(x,2r)\bigr)
    \le C_\mu\,\mu\bigl(B(x,r)\bigr)<\infty
\]
for every $x\in X$ and every $r>0$, where $B(x,r)$ is the ball centered at $x$ with radius $r$, i.e.  
\[
B(x,r)=\{y\in X:d(x,y)<r\}.
\]
\end{enumerate}
\end{definition}

We write
\[
A_r f(x) := \frac{1}{\mu(B(x,r))} \int_{B(x,r)} f(y)\, d\mu(y).
\]

The operators $A_r$, $r>0$, are called averaging operators.
They play the role of translations in Euclidean spaces.
Accordingly, we seek a condition weaker than the norm continuity under translations in Euclidean spaces
\cite{Garcia21,Krukowski,Xu}.
For the case of locally compact groups, see \cite{Feichtinger82,Feichtinger84,Krukowski}.
The idea of replacing translations by the averaging operators $A_r$, $r>0$, has recently been investigated in
\cite{Koshino23,Koshino24}.
This approach has the advantage that it does not rely on any group structure of the underlying space $X$.

This observation serves as the starting point of the present paper, which strengthens the result of Górka and Macios \cite{GM}.

\begin{theorem}[Kolmogorov--Riesz compactness in $L^1$ on spaces of homogeneous type]\label{L1-Kolmogrov-SHT}
Let $(X,d,\mu)$ be a space of homogeneous type, i.e.\ $d$ is a quasi-metric and $\mu$ is a doubling measure.
Assume that $A_r:L^1(X) \to L^1(X)$ is bounded uniformly over $r>0$.
A subset $\mathcal{F} \subset L^1(X)$ is relatively compact in $L^1(X)$ if and only if the following conditions hold:

\begin{enumerate}

\item \textbf{Uniform integrability{\rm:}}
\[
\sup_{f \in \mathcal{F}} \|f\|_{L^1(X)} < \infty.
\]
\item \textbf{Tightness{\rm:}}
There exists an increasing sequence of 
concentric balls $B_R$ of radius $R>0$ such that
\begin{equation}\label{eq:tight}
\lim_{R \to \infty}\left(
\sup_{f \in \mathcal{F}} \int_{X \setminus B_R} |f(x)|\, d\mu(x)\right)=0.
\end{equation}

\item \textbf{Vanishing local oscillation{\rm:}}
For all $\varepsilon > 0$, there exists $r>0$ such that
\begin{equation}\label{eq:Vlo}
\sup_{f \in \mathcal{F}} \|f - A_r f\|_{L^1(X)} < \varepsilon.
\end{equation}
\end{enumerate}

Then $\mathcal{F}$ is relatively compact in $L^1(X)$.
\end{theorem}
Upon reexamining the proof, we observe that the same argument yields
a characterization of relatively compact subsets of $L^p(X)$ for
$p>1$. It is worth noting that the proof in \cite{GM} relies on the
reflexivity of $L^p(X)$ to establish this characterization.

We shall present a new approach that avoids the use of reflexivity
and is therefore applicable to
$L^p(X)$ with $p\ge1$. In particular, we provide
a proof of the characterization of relatively compact subsets of
$L^1(X)$.

Theorem \ref{L1-Kolmogrov-SHT} complements the result of Górka and Macios \cite{GM}, who characterized relatively compact subsets of $L^p(X)$ for $1<p<\infty$.
Unlike the proofs in \cite{AU,GM}, our proof of Theorem \ref{L1-Kolmogrov-SHT} does not rely on the Hardy--Littlewood maximal operator $M$.
Instead, motivated by recent developments in \cite{Garcia21,Xu}, we circumvent this difficulty by employing the averaging operators $A_r$.
For comparison, Dymek and Górka overcame the corresponding difficulty by imposing an additional condition in \cite[Theorem~3.1(ii)]{DyGo}.
However, their argument depends on the specific structure of variable exponent Lebesgue spaces.

The primary objective of this paper is to extend Theorem \ref{L1-Kolmogrov-SHT} to the setting of general Banach function spaces.
To this end, we adopt the following terminology from \cite{SHYY17}.
Here and below we denote by $L^0(X)$ the set of all measurable functions.
\begin{definition}
Let $(X,d,\mu)$ be a space of homogeneous type. A Banach space
$\mathcal B(X)$ consisting of $\mu$-measurable functions on $X$ is said
to be a \emph{ball Banach space} if the following conditions are
satisfied:

\begin{enumerate}
\item[(i)] If $f\in\mathcal B(X)$, then $|f|<\infty$ almost everywhere on
$X$.

\item[(ii)] If $f,g\in L^0(X)$ satisfy $|f(x)|\le |g(x)|$ almost
everywhere on $X$ and $g\in\mathcal B(X)$, then $f\in\mathcal B(X)$ and
\[
\|f\|_{{\mathcal B}(X)}\le \|g\|_{{\mathcal B}(X)}.
\]

\item[(iii)] For every ball $B\subset X$, the characteristic function
$\chi_B
$ belongs to $\mathcal B(X)$.

\item[(iv)] For every ball $B\subset X$, there exists a positive
constant $C_B
$ such that
\begin{equation}\label{eq:260709-1}
\int_B |f(x)|\,d\mu(x)
\le C_B\|f\|_{{\mathcal B}(X)}
\end{equation}
for all $f\in\mathcal B(X)$.

\item[(v)] If $\{f_j\}_{j=1}^\infty\subset\mathcal B(X)$ satisfies
$0\le f_j\uparrow f$ almost everywhere, then
\[
\|f_j\|_{{\mathcal B}(X)}\uparrow\|f\|_{{\mathcal B}(X)}.
\]
\end{enumerate}
\end{definition}
Note that Minkowski's inequality holds in ball Banach function spaces,
as it follows from the triangle inequality for Bochner integrals
\cite[Lemma~98]{SaFaHa}.
We will frequently use the local embedding property of ball Banach function spaces.
That is, if $\{f_j\}_{j=1}^{\infty}$ converges to $f$ in $L^\infty(X)$
and all $f_j$ are supported in a common ball $B$, then
$\{f_j\}_{j=1}^{\infty}$ converges to $f$ in $\mathcal B(X)$.

Our main goal is to replace $L^1(X)$ by a more general Banach lattice
$\mathcal B(X)$.

Let $C_{\mathrm c}(X)$ denote the space of all continuous functions having bounded support.
We write $\widetilde{\mathcal B}(X)$ for the closure of $C_{\mathrm c}(X)$ in $\mathcal B(X)$.

Our main result may be viewed as a generalization of the theorem of Górka and Macios \cite{GM}.
Indeed, the special case $\mathcal B=L^p(X)$ with $1<p<\infty$ was proved there.

We suppose that the functions in $\mathcal B(X)$ enjoys a nice property:
\begin{definition}\label{defi:260722-1}
Let $\mathcal B(X)$ be a Banach lattice of measurable functions on a
measure space $(X,\mu)$. We say that the norm of $\mathcal B(X)$ is
\emph{absolutely continuous} if, for every $f\in\mathcal B(X)$ and every
sequence of measurable sets $\{E_j\}_{j=1}^{\infty}$ satisfying
\[
\chi_{E_j}(x)\to 0 \quad \text{almost everywhere on }X,
\]
we have
\[
\|f\chi_{E_j}\|_{{\mathcal B}(X)}\to 0
\qquad (j\to\infty).
\]
In this case, we also say that $\mathcal B(X)$ has an
\emph{absolutely continuous norm}.
The space $\mathcal{B}_{\rm a}$ stands for the set of all elements
having absolutely continuous norm.
\end{definition}
Motivated by the preceding definition, we introduce a local version of absolute continuity.
\begin{definition}\label{defi:260722-2}
Let $\mathcal B(X)$ be a Banach function space. We say that a function
$f\in\mathcal B(X)$ has a \emph{locally vanishing norm} if
\[
\lim_{r\to0}
\|f\chi_{B(x,r)}\|_{{\mathcal B}(X)}=0
\]
for all $x \in X$.
We denote by $\widetilde{\mathcal B}_0(X)$ the set of all functions in $\mathcal B(X)$
having a continuous norm.
\end{definition}
The following definition strengthens the notion of a locally vanishing norm by requiring the convergence to be uniform with respect to the center of the ball.
\begin{definition}\label{defi:260722-6}
Let $\mathcal B(X)$ be a Banach function space. We say that a function
$f\in\mathcal B(X)$ has a \emph{uniformly continuous norm} if
\[
\lim_{r\to0}
\sup_{x \in X}\|f\chi_{B(x,r)}\|_{{\mathcal B}(X)}=0.
\]
We denote by $\widetilde{\mathcal B}_{\rm c}(X)$ the set of all functions in $\mathcal B(X)$
having an  \emph{uniformly continuous norm}. 
\end{definition}
Definitions~\ref{defi:260722-1} and~\ref{defi:260722-2}
are natural. However, they are not sufficient for the approximation
argument.
In this paper, we will 
introduce notions analogous to
Definitions~\ref{defi:260722-1} and~\ref{defi:260722-2}.

Similarly, we define ${\mathcal B}_\infty(X)$ as the subspace of $\mathcal B(X)$
consisting of all functions $f\in\mathcal B(X)$ satisfying
\[
\lim_{r\to\infty}
\|f\chi_{B(x,r)^{\rm c}}\|_{{\mathcal B}(X)}=0.
\]

Finally,
in Section \ref{section:Examples},
we consider examples.
See \cite{Koshino24}
for the case of Lorentz spaces.

We state our characterization
of relatively compact  subset of $\widetilde{\mathcal B}(X)$.
\begin{theorem}[Kolmogorov--Riesz compactness in $\mathcal{B}(X)$
on spaces of homogeneous type]
\label{thm:main}
Let $(X,d,\mu)$ be a space of homogeneous type, i.e.\ $d$ is a quasi-metric and $\mu$ is a doubling measure.
Assume that $A_r:\mathcal{B} \to \mathcal{B}(X)$ is bounded uniformly over $r>0$.
A subset $\mathcal{F} \subset \widetilde{\mathcal{B}}$ is relatively compact in $\mathcal{B}(X)$ if and only if the following conditions hold{\rm:}

\begin{enumerate}
\item \textbf{Uniform integrability{\rm:}}
The class $\mathcal F$ is bounded, that is,
\[
\sup_{f \in \mathcal{F}} \|f\|_{\mathcal{B}} < \infty.
\]
\item \textbf{Tightness{\rm:}}
There exists an increasing sequence of concentric
balls
$B_R \subset X$ such that $B_R$
is a ball of radius $R>0$ and that
\[
\lim_{R \to \infty}\left(
\sup_{f \in \mathcal{F}} \|f\chi_{X \setminus B_R}\|_{\mathcal{B}}\right)=0.
\]

\item \textbf{Vanishing local oscillation{\rm:}}
For all $\varepsilon > 0$, there exists $r>0$ such that
\[
\sup_{f \in \mathcal{F}} \|f - A_r f\|_{\mathcal{B}} < \varepsilon.
\]
\end{enumerate}

Then $\mathcal{F}$ is relatively compact in ${\mathcal{B}}$.
\end{theorem}
Compared with recent works, Theorem \ref{thm:main} has the following key features.

\begin{itemize}

\item
We do not need to assume that \eqref{eq:260709-1} holds for every measurable set of finite measure.
Such an assumption is imposed in \cite{Koshino23,Koshino24}.

\item
We do not need to assume the continuity of the function
\[
(x,r)\in X\times(0,\infty)\mapsto \mu(B(x,r)).
\]
We do not need to assume
\[
\lim_{y \to x}
\mu(B(x,r) \triangle B(y,r))=0,
\]
where $A \triangle B
$ means the symmetric difference of sets $A$ and $B
$.
This assumption is also imposed in \cite{Koshino23,Koshino24}.

\item
We do not impose additional structural assumptions on the underlying space of homogeneous type $(X,d,\mu)$.
In particular, we allow the case $\mu(X)=\infty$, whereas \cite{BGGS} assumes that $\mu(X)<\infty$.
Some authors assumed that
\begin{equation}\label{eq:260710-1}
\inf_{x \in X}\mu(B(x,r))>0.
\end{equation}
See
\cite[Theorem 1.1(2)]{Koshino24},
for example about condition \eqref{eq:260710-1}.
See \cite[Proposition~3.2]{Koshino24} for the relationship between this condition and the topological properties of $X$.

Koshino assumed the Vitali property in \cite{Koshino23,Koshino24}, which is used there to control the averaging operators $A_r$.
Several authors imposed substitute conditions for this property; see, for example, \cite[Theorem~1.2(2)]{Garcia21} and \cite[Theorem~1(i)]{Xu}.
In contrast, our approach avoids these additional assumptions by using Theorem \ref{L1-Kolmogrov-SHT}.

\item
We do not restrict ourselves to particular classes of function spaces.
For example, Koshino considered Lorentz spaces in \cite{Koshino24}, while several authors studied variable exponent Lebesgue spaces \cite{DyGo,Xu}.
Our framework applies to general Banach function spaces.
\end{itemize}

Here and below we use the following convention in this paper.
Let $A,B \ge 0$.
Then $A \lesssim B$ means
that there exists a constant $C>0$
such that $A \le C B$,
where $C$ is usually independent of the functions we are considering.
The symbol $A \sim B$ stands for the two-sided inequality
$A \lesssim B \lesssim A$.

This paper is organized as follows:
Section \ref{s3} is devoted to the proof of
Theorem \ref{L1-Kolmogrov-SHT} and Theorem \ref{thm:main}.
In Section \ref{s4}, we characterize the space
$\widetilde{B}(X)$.
Section \ref{sec:RI} treats the case where
$\mathcal B(X)$ is rearrangement-invariant.
In Section \ref{sec:euclidean}, we specialize our results
to the Euclidean setting and relate them to known results.
Finally, Section \ref{section:Examples} presents several examples.

\section{Proofs}
\label{s3}

We need the following variant
of
Arzel\`a--Ascoli theorem on a totally bounded set.
\begin{lemma}\label{lem:260720-1}
Let $(X,d,\mu)$ be a space of homogeneous type. Let $K\subset X$ be
totally bounded and let $D\subset K$ satisfy
\[
K\subset \overline{D}.
\]
Suppose that $\{f_j\}_{j=1}^{\infty}\subset C(K)$ satisfies

\begin{enumerate}
\item[(i)] 
$\displaystyle
\sup_{j\in\mathbb N}\|f_j\|_{L^\infty(K)}<\infty;
$

\item[(ii)] $\{f_j\}_{j=1}^\infty$ is equicontinuous on $K$;

\item[(iii)] for every $x\in D$, the limit
\[
\lim_{j\to\infty}f_j(x)
\]
exists.
\end{enumerate}

Then there exists $f\in C(K)$ such that
\[
\lim_{j\to\infty}\|f_j-f\|_{L^\infty(K)}=0.
\]
\end{lemma}

We prove
Theorem~\ref{L1-Kolmogrov-SHT}
in Section~\ref{subsection:Proof of Theorem 1.2}
and
Theorem~\ref{thm:main}
in Section~\ref{subsection:Proof of Theorem 1.7}.

\subsection{Proof of Theorem \ref{L1-Kolmogrov-SHT}}
\label{subsection:Proof of Theorem 1.2}

Let $f \in L^1(X)$ and $r>0$.
Set
\[
K_r(x,y)=\frac{\max(r-d(x,y),0)}{r}
=\frac1r\int_0^r\chi_{B(x,s)}(y)\,ds 
\quad (x,y \in X),
\]
and define
\[
k_r(x)=\int_XK_r(x,y)\,d\mu(y)
=\frac1r\int_0^r\mu(B(x,s))\,ds 
\quad (x \in X).
\]
Let
\[
B_rf(x)
=\frac1{k_r(x)}
\int_XK_r(x,y)f(y)\,d\mu(y)
\quad (x \in X).
\]

For \(x\in X\), we have
\[
\begin{aligned}
f(x)-B_rf(x)
&=
\frac1{k_r(x)}
\int_XK_r(x,y)(f(x)-f(y))\,d\mu(y)\\
&=
\frac1{rk_r(x)}
\int_0^r
\int_{B(x,s)}(f(x)-f(y))\,d\mu(y)\,ds .
\end{aligned}
\]
Meanwhile, from the definition of $A_s$,
\[
\int_{B(x,s)}(f(x)-f(y))\,d\mu(y)
=
\mu(B(x,s))(f(x)-A_sf(x))
\]
for any $s>0$ and $x \in X$,
and consequently,
\[
|f(x)-B_rf(x)|
\leq
\frac1{rk_r(x)}
\int_0^r
\mu(B(x,s))
|f(x)-A_sf(x)|\,ds .
\]

Since \(\mu\) is doubling, 
\[
k_r(x)
=
\frac1r\int_0^r\mu(B(x,s))\,ds
\sim
\mu(B(x,r)),
\]
where the implicit constants are independent of \(x\) and \(r\). 
This, together with the doubling property of $\mu$, implies that for every $r>0$ there exists a constant $c=c(r)>0$ such that
\begin{equation}\label{eq:260721-1}
\inf_{x \in B_R} k_r(x) \ge c.
\end{equation}
Hence, $1/k_r$ is a bounded Lipschitz function.

Moreover,
\[
\mu(B(x,s))\leq \mu(B(x,r)),
\qquad 0<s<r .
\]
Therefore,
\[
|f(x)-B_rf(x)|
\lesssim
\frac1r
\int_0^r
|f(x)-A_sf(x)|\,ds .
\]

Taking the norm in a Banach function space \(\mathcal B\) and using
Minkowski's inequality, we obtain
\[
\|f-B_rf\|_{{\mathcal B}(X)}
\lesssim
\frac1r
\int_0^r
\|f-A_sf\|_{{\mathcal B}(X)}\,ds .
\]
Thus, the family \(\{B_r\}_{r>0}\) satisfies the same approximation
property as \(\{A_r\}_{r>0}\). Hence, 
for the proof of the relative
compactness of \(\mathcal F\), it is enough to prove the relative
compactness of
\[
B_r\mathcal F
=
\{B_rf:f\in\mathcal F\}
\]
for every fixed \(r>0\).

Recall that $X$ is separable
(see \cite[Lemma 2.10]{KMYZ}).
Let \(D\subset X\) be a countable dense subset, and let
$
\{f_j\}_{j=1}^{\infty}\subset\mathcal F
$
be an arbitrary sequence.
By a diagonal argument, after passing to a subsequence, we may assume that
\[
\int_XK_r(x,y)f_j(y)\,d\mu(y)
\]
converges for every \(x\in D\).

We may further assume that all \(f_j\) are supported in a fixed bounded
set $B_r$ due to \eqref{eq:tight} without breaking \eqref{eq:Vlo}.
In fact, \eqref{eq:tight}
allows us to assume
replace $f_j$ with
$\chi_{B_{R+r}}f_j$ with $R \gg 1$.

Consider the kernel
\[
\Phi_r(x,y)
=
\frac{K_r(x,y)}{k_r(x)} .
\]
For fixed \(r>0\), the function \(x\mapsto\Phi_r(x,y)\) is Lipschitz
continuous, and its Lipschitz constant is bounded independently of
\(y\). 
Furthermore,
since 
\[
|K_r(x,y)-K_r(x',y)|
\le
\frac{d(x,x')}{r}\chi_{B_{R+2r}}(y)
\]
for all $x,x' \in B_R$ satisfying $d(x,x')<r$,
\[
|k_r(x)-k_r(x')|
\le
\frac{d(x,x')}{r}\mu(B_{R+2r}).
\]
Thus, $k_r$ is Lipschitz on $U$.
Finally, by the doubing condition on $\mu$
and \eqref{eq:260721-1},
$k_r$ is bounded from below by a constant depending
only on $B_{R+2r}$ and $r$.
Therefore, by the uniform \(L^1\)-boundedness of \(\mathcal F\),
we have
\begin{align*}
|B_rf_j(x)-B_rf_j(z)|
=
\left|
\int_X
(\Phi_r(x,y)-\Phi_r(z,y))
f_j(y)\,d\mu(y)
\right|
\lesssim
d(x,z)\|f_j\|_{L^1(X)}
\lesssim
d(x,z).
\end{align*}
Hence, \(\{B_rf_j\}_{j=1}^\infty\) is an equicontinuous and uniformly bounded family
on bounded subsets of \(X\).

Since \(D\) is dense in \(X\), the convergence on \(D\), together with the
local equicontinuity mentioned above, 
Lemma \ref{lem:260720-1} guarantees
that a further
subsequence converges uniformly on 
arbitrary bounded subsets of \(X\). 
Since $\{f_j\}_{j=1}^\infty$
is supported on $B_{R}$
and $\mu(B_{R})<\infty$,
\[
\lim_{j \to \infty}
\|B_rf_j-B_rf\|_{L^1(X)}=0.
\]
Therefore,
$B_r\mathcal F$
is relatively compact in $L^1(X)$.

Applying the argument from the proof of
\cite[Theorem~6.44]{Sawano18},
we obtain that
\[
\mathcal F^*=
\left\{
f\in L^1(X):
f=\lim_{j\to\infty}B_{2^{-k_j}}f_j
\text{ in }L^1(X)
\;
\mbox{ for some }
k_1<k_2<\cdots,
\;
f_1,f_2,\ldots\in\mathcal F
\right\}
\]
is compact because each
$B_{2^{-k_j}}\mathcal F$
is compact in $L^1(X)$.
Since
$\mathcal F\subset\mathcal F^*$,
we conclude that
$\mathcal F$
is relatively compact in $L^1(X)$.

\subsection{Proof of Theorem \ref{thm:main}}
\label{subsection:Proof of Theorem 1.7}

The proof of Theorem~\ref{thm:main} relies on
Theorem~\ref{L1-Kolmogrov-SHT},
which guarantees that every sequence in $\mathcal{F}$
admits a subsequence converging in $L^1(X)$.

Let $r\in(0,1)$. Then
\[
A_r\mathcal{F}:=\{A_r f\,:\,f\in\mathcal{F}\}
\]
inherits the uniform integrability and tightness properties of
$\mathcal{F}$.
By the vanishing local oscillation property, it remains to show that
$A_r\mathcal{F}$ is relatively compact in $\mathcal B(X)$.

We first note that
\[
\lim_{R\to\infty}
\left(
\sup_{f\in\mathcal{F}}
\|\chi_{X\setminus B_R}A_r f\|_{\mathcal B}
\right)=0.
\]
Therefore, by an argument similar to that in the proof of
Theorem~\ref{L1-Kolmogrov-SHT},
it suffices to prove that
\[
A_r\mathcal{F}|_K
:=
\{\chi_K A_r f\,:\,f\in\mathcal{F}\}
\]
is relatively compact in $\mathcal B(X)$ for every bounded measurable set
$K$.
If necessary, we may further
assume that $K$ is bounded, by approximating it with
an increasing sequence of bounded measurable sets.

Suppose that we have a countable sequence
$\{f_j\}_{j=1}^\infty$ in $\mathcal{F}$.
Write
\begin{equation}\label{K dagger}
K^\dagger=\{y \in K\,:\,
d(y,K)<1\}.
\end{equation}
This sequence is bounded in $\mathcal B(X)$.
Since $\mathcal B(X)$ is a ball Banach function space, $\mathcal B(X)$ is embedded into
$L^1(X)$ locally.
Theorem \ref{L1-Kolmogrov-SHT} allows
us to assume that
$\{f_j\}_{j=1}^\infty$
converges
to a function $F$
in $L^1(K^\dagger)$.

We enhance the local convergence in $L^1(X)$ to convergence
in $\mathcal B$.
Since
$\{f_j\}_{j=1}^\infty$
converges to a function $F$
in $L^1(K^\dagger)$,
we have
\[
\int_{B(x,r)} f_j(z)\,d\mu(z)
\longrightarrow
\int_{B(x,r)} F(z)\,d\mu(z)
\]
for every $x\in K$ and $r>0$.
Moreover, the convergence is uniform with respect to
$x\in K$ and $r>0$.
Indeed,
\begin{align*}
\left|
\int_{B(x,r)}f_j(z)\,d\mu(z)
-
\int_{B(x,r)}F(z)\,d\mu(z)
\right|
&=
\left|
\int_{B(x,r)}(f_j(z)-F(z))\,d\mu(z)
\right| \\
&\le
\int_{B(x,r)}|f_j(z)-F(z)|\,d\mu(z) \\
&\le
\int_{K^\dagger}|f_j(z)-F(z)|\,d\mu(z) \\
&=
\|f_j-F\|_{L^1(K^\dagger)}.
\end{align*}
The right-hand side tends to zero as $j\to\infty$,
independently of $x\in K$ and $r>0$.

Recall that
$\mathcal F$ is bounded in $\mathcal B(X)$.
Since $\mu(K)<\infty$ and
$A_r\mathcal F|_K$ is bounded in $L^\infty(X)$,
\[
\lim_{j \to \infty}
\|\chi_KA_rf_j-\chi_KA_rF\|_{\mathcal B(X)}=0.
\]
Therefore,
$A_r\mathcal{F}|_K$
is compact in $\mathcal B(X)$,
and consequently so is
$\mathcal F$.

\section{A characterization of $\widetilde{\mathcal B}(X)$}
\label{s4}
We denote by
$
\widetilde{\mathcal B}(X)
$
the closure of $C_{\mathrm c}(X)$ in $\mathcal B(X)$.
In Section \ref{s4}, we characterize
$\widetilde{\mathcal B}(X)$.
The characterization is based on two conditions:
approximation by dyadic averaging operators and vanishing at infinity.

We first fix a system of Christ dyadic cubes.
More precisely, there exist constants
$0<\delta<1$,
$C_0>0$,
and $C_1>0$,
and collections
\[
\mathcal D_k=\{Q_\alpha^k:\alpha\in I_k\},
\qquad k\in\mathbb Z,
\]
of open subsets of $X$ satisfying the following properties
\cite[Theorem 11]{Christ}.

\begin{definition}[Christ dyadic cubes]
The sets $Q_\alpha^k$ are called the dyadic cubes of generation $k$,
and $\mathcal D_k$ is called the $k$th dyadic generation.
They satisfy:

\begin{enumerate}
\item
For every $k\in\mathbb Z$,
\[
\mu\left(
X\setminus
\bigcup_{\alpha\in I_k}Q_\alpha^k
\right)=0.
\]

\item
If $\ell\ge k$, then for every
$Q_\beta^\ell\in\mathcal D_\ell$ and
$Q_\alpha^k\in\mathcal D_k$,
either
$
Q_\beta^\ell\subset Q_\alpha^k,
$
or
$
Q_\beta^\ell\cap Q_\alpha^k=\varnothing .
$

\item
There exist $C_0,C_1>0$ such that,
for every $(k,\alpha)$, there exists
$z_\alpha^k\in X$ such that
\begin{equation}\label{eq:260722-4}
B(z_\alpha^k,C_0\delta^k)
\subset
Q_\alpha^k
\subset
B(z_\alpha^k,C_1\delta^k).
\end{equation}

\item
Each cube has a unique parent. Namely, for every
$Q_\beta^{k+1}\in\mathcal D_{k+1}$,
there exists a unique
$Q_\alpha^k\in\mathcal D_k$
such that
\[
Q_\beta^{k+1}\subset Q_\alpha^k .
\]
\end{enumerate}
\end{definition}

For $Q=Q^k_\alpha \in \mathcal D_k$,
we write $B_Q$
for the ball $B(z^k_\alpha,C_0\delta^k)$
as in \eqref{eq:260722-4}.
For each $k\in\mathbb Z$,
we let
$
\{\psi_Q\}_{Q\in\mathcal D_k}
$
be a family of Lipschitz continuous functions forming a partition of unity
subordinate to the dyadic cubes in $\mathcal D_k$.
More precisely,
\begin{equation}\label{eq:260712-1}
0\le \psi_Q\le1,
\qquad
\operatorname{supp}\psi_Q\subset C B_Q ,
\end{equation}
where $C>1$ is independent of $k$ and $Q$, and
\begin{equation}\label{eq:260712-2}
\sum_{Q\in\mathcal D_k}\psi_Q=1.
\end{equation}
Remark that
\begin{equation}\label{eq:260722-5}
\sum_{Q \in {\mathcal D}_k}\chi_{C B_Q}
\lesssim 1
\end{equation}
due to the doubling property of $\mu$
with the implicit constant independent of $k$.

The {\it uniform
boundedness of dyadic averaging operators}
is said to hold, if
\[
f \mapsto
\widetilde{T_k}f
=
\sum_{Q\in\mathcal D_k}
\left(
\frac1{\mu(Q)}
\int_Q f(y)\,d\mu(y)
\right)
\chi_{C B_Q}
\]is bounded on
$\mathcal B(X)$ uniformly over $k \in {\mathbb Z}$.
\begin{theorem}[Approximation by dyadic averaging operators]
\label{thm:dyadic-approximation}
Let $(X,d,\mu)$ be a space of homogeneous type and let
$\mathcal B(X)$ be a ball Banach function space.

Assume that 
the uniform boundedness of dyadic averaging operators
holds.

For each $k\in\mathbb Z$, define
\[
T_kf
:=
\sum_{Q\in\mathcal D_k}
\left(
\frac1{\mu(Q)}
\int_Q f(y)\,d\mu(y)
\right)
\psi_Q .
\]

Then
\[
\lim_{k\to\infty}
\|T_kf-f\|_{\mathcal B(X)}
=
0
\]
for every
$f\in\widetilde{\mathcal B}(X)$.
\end{theorem}
Prior to the proof, we note that the sum defining
$T_k$ is locally finite by \eqref{eq:260722-5}.
Since $\widetilde{T_k}$ is bounded on $\mathcal B(X)$,
the operator $T_k$ is also a bounded linear operator
on $\mathcal B(X)$.
\begin{proof}
For each $Q\in\mathcal D_k$, 
we abbreviate
\[
A_Q(f)
=
\frac1{\mu(Q)}
\int_Q f(y)\,d\mu(y).
\]
Then
\[
T_kf
=
\sum_{Q\in\mathcal D_k}
A_Q(f)\psi_Q.
\]

Fix $x\in X$.
By the triangle inequality and 
\eqref{eq:260712-1},
\[
|T_kf(x)|
\le
\sum_{Q \in \mathcal{D}_k} |A_Q(f)|\chi_{C B_Q}(x)
\le
\widetilde{T_k}[|f|](x).
\]
The 
uniform boundedness of dyadic averaging operators
gives
\[
\|T_kf\|_{{\mathcal B}(X)}
\le
\|\widetilde{T_k}[|f|]\|_{{\mathcal B}(X)}
\le
C\|\,|f|\,\|_{{\mathcal B}(X)}
=
C\|f\|_{{\mathcal B}(X)},
\]
where $C$ is independent of $k$.

Let
$f\in C_{\mathrm c}(X)$.
Since $f$ is uniformly continuous, for every
$\varepsilon>0$ there exists $\eta>0$ such that
$
|f(x)-f(y)|<\varepsilon 
$
whenever
$
d(x,y)<\eta$.

Let $C$ be the constant appearing in \eqref{eq:260712-1}.
Let $a>1$.
Choose $k$ sufficiently large so that every cube in
$\mathcal D_k$ has diameter less than $\eta/(aC)$.
Furthermore, by taking $a$ sufficiently large, we may assume that
\[
d(x,y)<\eta
\]
whenever $x\in C B_Q$ and $y\in Q$.

If $\psi_Q(x)\neq0$, then $x\in C B_Q$ and therefore
$d(x,y)<\eta$
for any
$y\in Q$.
Hence
\[
|A_Q(f)-f(x)|
\le
\frac1{\mu(Q)}
\int_Q |f(y)-f(x)|\,d\mu(y)
<
\varepsilon .
\]

Using \eqref{eq:260712-2},
we obtain
\[
\begin{aligned}
|T_kf(x)-f(x)|
&=
\left|
\sum_{Q \in \mathcal{D}_k}
(A_Q(f)-f(x))\psi_Q(x)
\right|
\le
\sum_{Q \in \mathcal{D}_k}
|A_Q(f)-f(x)|\psi_Q(x)
<
\varepsilon .
\end{aligned}
\]

Thus, we conclude
\[
\lim_{k \to \infty}
\|T_kf-f\|_{L^\infty(X)}
=0 .
\]

Since $f$ has compact support and $T_kf$ has support contained in a
fixed bounded neighborhood of $\operatorname{supp}f$
as long as $k \gg 1$, 
we conclude
from
the local
embedding property of the ball Banach function space that $T_k f$ converges back to $f$
in ${\mathcal B}(X)$.

Finally, let
$f\in\widetilde{\mathcal B}(X)$
be arbitrary.
For every $\varepsilon>0$ choose
$g\in C_{\mathrm c}(X)$ such that
$
\|f-g\|_{{\mathcal B}(X)}<\varepsilon .
$

Then
\begin{align*}
\|T_kf-f\|_{{\mathcal B}(X)}
&\le
\|T_k(f-g)\|_{{\mathcal B}(X)}
+
\|T_kg-g\|_{{\mathcal B}(X)}
+
\|g-f\|_{{\mathcal B}(X)}\\
&\le
(C+1)\varepsilon
+
\|T_kg-g\|_{{\mathcal B}(X)}.
\end{align*}

Since $g \in C_{\rm c}(X)$,
we have
$
\|T_kg-g\|_{{\mathcal B}(X)}\to0 .
$
Therefore,
\[
\limsup_{k\to\infty}
\|T_kf-f\|_{{\mathcal B}(X)}
\le
(C+1)\varepsilon .
\]

Since $\varepsilon>0$ is arbitrary,
$\displaystyle
\lim_{k\to\infty}
\|T_kf-f\|_{{\mathcal B}(X)}
=
0 .
$
\end{proof}

Motivated by Theorem \ref{thm:dyadic-approximation},
we present the following definition.
\begin{definition}
We define
\[
\mathcal B_{\rm loc}(X)
=
\left\{
f\in\mathcal B(X):
\lim_{k\to\infty}
\|T_kf-f\|_{\mathcal B(X)}
=
0
\right\}.
\]
\end{definition}

Thus, $\mathcal B_0(X)$ consists of all functions which are
approximated in $\mathcal B(X)$ by the dyadic averaging operators.

\begin{definition}\label{defi:260722-5}
Let $x_0\in X$ be fixed.
We define
\[
\mathcal B_\infty(X)
=
\left\{
f\in\mathcal B(X):
\lim_{R\to\infty}
\left\|
f\chi_{X\setminus B(x_0,R)}
\right\|_{\mathcal B(X)}
=
0
\right\}.
\]
The space $\mathcal B_\infty(X)$ is called the space of functions
vanishing at infinity in the $\mathcal B(X)$-norm.

The above definition is independent of the choice of the reference point
$x_0$.
\end{definition}

\begin{proposition}
\label{prop:Binfty-closed}
The space $\mathcal B_\infty(X)$ is a closed subspace of
$\mathcal B(X)$.
\end{proposition}

\begin{proof}
Let $\{f_j\}_{j=1}^{\infty}$ be a sequence in
$\mathcal B_\infty(X)$ such that
$f_j\to f$
in $\mathcal B(X)$
as $j \to \infty$.
We aim to show that
$f\in\mathcal B_\infty(X)$.

Let $\varepsilon>0$.
Choose $j$ sufficiently large so that
$
\|f-f_j\|_{\mathcal B(X)}
<
\varepsilon .
$
Since
$f_j\in\mathcal B_\infty(X)$,
there exists $R>0$ such that
$
\|f_j\chi_{X\setminus B(x_0,R)}\|_{\mathcal B(X)}
<
\varepsilon .
$
Therefore,
\[
\begin{aligned}
\|f\chi_{X\setminus B(x_0,R)}\|_{\mathcal B(X)}
&\le
\|(f-f_j)\chi_{X\setminus B(x_0,R)}\|_{\mathcal B(X)}
+
\|f_j\chi_{X\setminus B(x_0,R)}\|_{\mathcal B(X)}
\\
&\le
\|f-f_j\|_{\mathcal B(X)}
+
\|f_j\chi_{X\setminus B(x_0,R)}\|_{\mathcal B(X)}
\\
&<
2\varepsilon .
\end{aligned}
\]
Since $\varepsilon>0$ is arbitrary
and $\|f\chi_{X\setminus B(x_0,R)}\|_{\mathcal B(X)}$ is monotone in $R$,
\[
\lim_{R\to\infty}
\|f\chi_{X\setminus B(x_0,R)}\|_{\mathcal B(X)}
=
0 .
\]
Hence
$f\in\mathcal B_\infty(X)$.
\end{proof}

\begin{proposition}
\label{thm:B0-Binfty}Assume that 
the uniform boundedness of dyadic averaging operators
holds.
Then
$
\mathcal B_0(X)\cap\mathcal B_\infty(X)
\subset\widetilde{\mathcal B}(X).
$
\end{proposition}

\begin{proof}
Let $\varepsilon>0$
and
$
f\in
\mathcal B_0(X)\cap\mathcal B_\infty(X).
$

Since
$f\in\mathcal B_\infty(X)$,
there exists $R>0$ such that
\[
\|f\chi_{X\setminus B(x_0,R)}\|_{\mathcal B(X)}
<
\varepsilon .
\]

Choose $k_0\in\mathbb N$ so that the 
unique dyadic cube
$
Q_0\in\mathcal D_{-k_0}
$
containing $x_0$ satisfies
$
B(x_0,R)\subset Q_0 .
$
Then
\[
\chi_{Q_0}f-f
=
-f\chi_{X\setminus Q_0},
\]
and consequently,
\[
\|\chi_{Q_0}f-f\|_{\mathcal B(X)}
\le
\|f\chi_{X\setminus B(x_0,R)}\|_{\mathcal B(X)}
<
\varepsilon .
\]

Since
$f\in\mathcal B_0(X)$,
there exists $k_1\ge -k_0$ such that
\[
\|T_kf-f\|_{\mathcal B(X)}
<
\varepsilon ,
\qquad k\ge k_1 .
\]

For such $k$, every cube in $\mathcal D_k$ is either contained in
$Q_0$ or disjoint from $Q_0$.
Hence
\[
T_k(\chi_{Q_0}f)
=
\sum_{Q\in\mathcal D_k, Q \subset Q_0}
\left(
\frac1{\mu(Q)}
\int_Q f(y)\,d\mu(y)
\right)
\psi_Q .
\].

Therefore,
letting $C$ be the uniform bound of $T_k$
acting on $\mathcal B(X)$, we obtain
\[
\begin{aligned}
\|T_k(\chi_{Q_0}f)-T_k f\|_{\mathcal B(X)}
\le C
\|\chi_{Q_0}f-f\|_{\mathcal B(X)}
<
C\varepsilon .
\end{aligned}
\]

Combining this with the previous estimate gives
\[
\begin{aligned}
\|T_k(\chi_{Q_0}f)-f\|_{\mathcal B(X)}
\le
\|T_k(\chi_{Q_0}f)-T_k f\|_{\mathcal B(X)}
+
\|T_k f-f\|_{\mathcal B(X)}
<
(1+C)\varepsilon .
\end{aligned}
\]

Finally,
observe that
$T_k(\chi_{Q_0}f)\in C_{\mathrm c}(X)$.
Hence $f$ belongs to the closure of $C_{\mathrm c}(X)$ in
$\mathcal B(X)$.
Therefore,
$
f\in\widetilde{\mathcal B}(X).
$
\end{proof}

The following theorem characterizes the closure of compactly supported functions
in terms of the local and global vanishing properties of the norm.
\begin{theorem}[Characterization]
\label{thm:characterization}
Let $(X,d,\mu)$ be a space of homogeneous type.
Assume that 
the uniform boundedness of dyadic averaging operators
holds.
Then
\begin{equation}\label{eq:260712-3}
\widetilde{\mathcal B}(X)
=
\mathcal B_0(X)
\cap
\mathcal B_\infty(X).
\end{equation}
\end{theorem}

\begin{proof}
We first prove the inclusion
\[
\widetilde{\mathcal B}(X)
\subset
\mathcal B_0(X)\cap\mathcal B_\infty(X).
\]

Let
$f\in\widetilde{\mathcal B}(X)$
to this end.

By Theorem~\ref{thm:dyadic-approximation},
\[
\lim_{k\to\infty}
\|T_kf-f\|_{\mathcal B(X)}
=
0,
\]
and therefore
$
f\in\mathcal B_0(X).
$

Moreover,
$
C_{\mathrm c}(X)\subset\mathcal B_\infty(X).
$
Since $\mathcal B_\infty(X)$ is closed by
Proposition~\ref{prop:Binfty-closed},
we obtain
\[
\widetilde{\mathcal B}(X)
\subset
\mathcal B_\infty(X).
\]

Hence,
$
\widetilde{\mathcal B}(X)
\subset
\mathcal B_0(X)
\cap
\mathcal B_\infty(X).
$

The reverse inclusion follows from
Proposition~\ref{thm:B0-Binfty}.
Therefore, \eqref{eq:260712-3}.
\end{proof}

We compare the assumptions
on $A_r$ and $\widetilde{T_k}$ in this paper.
\begin{remark}
If the operator $A_r$, $r>0$,
is uniformly bounded on $\mathcal B$,
then $\widetilde{T_k}$ is uniformly bounded
on $\mathcal B$ since
$|\widetilde{T_k}f| \lesssim A_{C r}(|f|)$
due to \eqref{eq:260722-5}
for any $f \in \mathcal B$.
\end{remark}

By using
Theorems
\ref{thm:main}
and
\ref{thm:characterization},
we can get a new compactness criterion.
\begin{corollary}[Kolmogorov--Riesz compactness in $\mathcal{B}(X)$
on spaces of homogeneous type]
Let $(X,d,\mu)$ be a space of homogeneous type, i.e.\ $d$ is a quasi-metric and $\mu$ is a doubling measure.
Assume that $A_r:\mathcal{B}(X) \to \mathcal{B}(X)$ is bounded uniformly over $r>0$.
A subset $\mathcal{F} \subset \mathcal B_0(X)
\cap
\mathcal B_\infty(X).
$ is relatively compact in $\mathcal{B}(X)$ if and only if the following conditions hold:

\begin{enumerate}
\item \textbf{Uniform integrability{\rm:}}
$\mathcal F$ is bounded on $\mathcal B(X)$, that is,
\[
\sup_{f \in \mathcal{F}} \|f\|_{\mathcal{B}} < \infty.
\]
\item \textbf{Tightness{\rm:}}
There exists an increasing sequence of concentric balls $B_R$ with radius $R$ such that
\[
\lim_{R \to \infty}\left(
\sup_{f \in \mathcal{F}} \|f\chi_{X \setminus B_R}\|_{\mathcal{B}}\right)=0.
\]

\item \textbf{Vanishing local oscillation{\rm:}}
For all $\varepsilon > 0$, there exists $r>0$ such that
\[
\sup_{f \in \mathcal{F}} \|f - A_r f\|_{\mathcal{B}} < \varepsilon.
\]
\end{enumerate}

Then $\mathcal{F}$ is relatively compact in ${\mathcal{B}}$.
\end{corollary}

\section{Rearrangement invariant spaces}\label{sec:RI}
Theorem 
\ref{thm:main} immediately yields compactness criteria for many concrete Banach function spaces. 
However, when $\mathcal{B}(X)$
is rearrangement invariant,
we can loosen the assumption
as we will see.
\begin{definition}[Rearrangement-invariant space]
Let \((\Omega,\mu)\) be a measure space. 
For a measurable function \(f\), its distribution function is defined by
\[
d_f(\lambda)
=
\mu\bigl(\{x\in\Omega: |f(x)|>\lambda\}\bigr),
\qquad \lambda>0.
\]
The decreasing rearrangement of \(f\) is defined by
\[
f^*(t)
=
\inf\{\lambda>0:d_f(\lambda)\le t\},
\qquad t>0 .
\]

A Banach function space \(X(\Omega)\) is called
\emph{rearrangement-invariant} if, for any measurable functions
\(f\) and \(g\) satisfying
\[
g^*(t)\le f^*(t),
\qquad t>0,
\]
the condition
\[
f\in X(\Omega)
\]
implies
\[
g\in X(\Omega)
\]
and
\[
\|g\|_{X(\Omega)}
\le
\|f\|_{X(\Omega)}.
\]
In particular, the norm of \(X(\Omega)\) depends only on the
distribution function of a function. More precisely,
\[
f^*=g^*
\quad\Longrightarrow\quad
\|f\|_{X(\Omega)}
=
\|g\|_{X(\Omega)} .
\]
\end{definition}

We postulate some conditions on
the underlying space.
\begin{definition}[Resonant measure space]
A totally \(\sigma\)-finite measure space \((R,\mu)\) is called
\emph{resonant} if, for every
\[
f,g\in L^0(R,\mu),
\]
one has
\[
\int_0^\infty f^*(t)g^*(t)\,dt
=
\sup_h
\int_R |f(x)h(x)|\,d\mu(x),
\]
where the supremum is taken over all measurable functions \(h\)
equimeasurable with \(g\).
\end{definition}

Throughout this paper, we assume that the underlying measure space is
resonant. This assumption is needed in order to use the standard
duality and interpolation theory for rearrangement-invariant spaces.
For example, the measure space
\[
X={\mathbb Z},
\]
equipped with the measure
\[
\mu(n)=\frac{n^4+2}{n^4+1},
\qquad n\in{\mathbb Z},
\]
is not resonant and is therefore excluded
\cite[Chapter 2, Theorem 2.7]{BS}.
This restriction is not too strong.

We recall the fundamental definition
of Luxemberg \cite{BS}.
\begin{definition}[Representation space]
Let \(X(\Omega)\) be a rearrangement-invariant space on a measure space
\((\Omega,\mu)\). The \emph{representation space} of \(X(\Omega)\) is
the rearrangement-invariant space
\[
\overline{X}(0,\mu(\Omega))
\]
on the interval \((0,\mu(\Omega))\) defined by
\[
f\in X(\Omega)
\quad\Longleftrightarrow\quad
f^*\in \overline{X}(0,\mu(\Omega)),
\]
with the norm relation
\[
\|f\|_{X(\Omega)}
=
\|f^*\|_{\overline{X}(0,\mu(\Omega))}.
\]
Here \(f^*\) denotes the decreasing rearrangement of \(f\).
\end{definition}

The following theorem characterizes the relative compactness
in rearrangement-invariant spaces.
Compared to Theorem \ref{thm:main},
it is closer to the classical Kolmogorov--Riesz compactness criterion for
Lebesgue spaces to the setting of general rearrangement-invariant Banach
function spaces.
\begin{theorem}[Kolmogorov--Riesz theorem for rearrangement-invariant spaces]
\label{thm:kolmogorov-riesz-ri}
Let $X(\Omega)$ be a rearrangement-invariant Banach function space, and let
$X(\Omega)_{\mathrm a}$ denote its absolutely continuous part.
Suppose that $\mathcal F\subset X(\Omega)_{\mathrm a}$ is bounded.

Then $\mathcal F$ is relatively compact in $X(\Omega)$ if and only if the following conditions are satisfied:

\begin{enumerate}
\item
For every $R>0$, the family
\[
\{f\chi_{B(x_0,R)}:f\in\mathcal F\}
\]
is relatively compact in $L^1(B(x_0,R))$.

\item
$\mathcal F$ vanishes uniformly at infinity,
that is,
\[
\lim_{R\to\infty}
\sup_{f\in\mathcal F}
\|f\chi_{\Omega\setminus B(x_0,R)}\|_{\mathcal{B}(X)}=0.
\]

\item
$\mathcal F$ vanishes uniformly on sets of small measure,
that is,
\[
\lim_{\varepsilon\to0}
\sup_{f\in\mathcal F}
\sup_{\substack{E\subset\Omega\\ \mu(E)<\varepsilon}}
\|f\chi_E\|_{\mathcal{B}(X)}=0.
\]
\end{enumerate}
\end{theorem}

\begin{proof}
We prove only the sufficiency.

Let $\{f_j\}_{j=1}^{\infty}\subset\mathcal F$.
By condition (2), it is enough to prove the convergence on a fixed sufficiently
large ball. Hence, without loss of generality, we may assume that
$\Omega=B(x_0,1)$.

By condition (1), after passing to a subsequence if necessary,
we may assume that
\[
f_j\to f
\quad\text{in }L^1(\Omega).
\]
Passing to a further subsequence, we may also assume that
\[
f_j(x)\to f(x)
\quad\text{for almost every }x\in\Omega.
\]
The Fatou property shows that $f \in \mathcal B(X)$.

Since $X(\Omega)$ is a rearrangement-invariant Banach function space,
\[
f^*(t)
\le
\|\chi_{(0,t)}f^*\|_{\overline{{\mathcal B}(X)}}
\le
\frac{\|f\|_{\mathcal{B}(X)}}{\varphi_{\mathcal{B}(X)}(t)}
\quad(t>0),
\]
where $\varphi_{\mathcal{B}(X)}$ denotes the fundamental function of $\mathcal{B}(X)$.
Therefore, by truncating the functions if necessary, we may assume that
the sequence $\{f_j\}_{j=1}^{\infty}$ is uniformly bounded.

Let $\varepsilon>0$.
By Egorov's theorem, there exists a measurable set
$A\subset\Omega$ such that
\[
\mu(\Omega\setminus A)<\varepsilon
\]
and $\{f_j\}_{j=1}^{\infty}$ converges uniformly to $f$ on $A$.

Since the convergence is uniform on $A$, we have
\[
\|(f_j-f)\chi_A\|_{\mathcal{B}(X)}\longrightarrow0.
\]
On the other hand, condition (3) yields
\[
\sup_{j\ge1}
\|(f_j-f)\chi_{\Omega\setminus A}\|_{\mathcal{B}(X)}
\]
is arbitrarily small whenever $\varepsilon$ is sufficiently small.
Combining these two estimates, we conclude that
\[
\|f_j-f\|_{\mathcal{B}(X)}\longrightarrow0.
\]
Hence $\mathcal F$ is relatively compact in ${\mathcal B}(\Omega)$.
\end{proof}

Before proceeding further, we recall an elementary consequence of
the interpolation theory for rearrangement-invariant spaces.
\begin{remark}
Let \(T\) be a linear operator which is bounded on both
\(L^1(X)\) and \(L^\infty(X)\); namely,
\[
\|Tf\|_{L^1(X)}\le C_1\|f\|_{L^1(X)}
\quad (f \in L^1(X)),
\qquad
\|Tf\|_{L^\infty(X)}
\le C_\infty\|f\|_{L^\infty(X)} \quad (f \in L^\infty(X)).
\]
This leads to  the \(K\)-functional estimate:
\[
K(Tf,L^1(X),L^\infty(X))
\le \max(C_1,C_\infty) K(f,L^1(X),L^\infty(X))
\quad (f \in L^1(X)+L^\infty(X)).
\]
Since
\[
K(f,L^1(X),L^\infty(X))
=
\int_0^t f^*(s)\,ds  \quad (t>0)
\]
according to
\cite[Chapter 2, Theorem 6.2]{BS},
it follows that
\[
\int_0^t (Tf)^*(s)\,ds
\le
C\int_0^t f^*(s)\,ds ,
\qquad t>0 .
\]
Hence, in the notation of
\cite[Chapter 2, Definition 3.5]{BS},
\[
Tf\preceq C f .
\]
Therefore, by
\cite[Chapter 2, Theorem 4.6]{BS},
\(T\) is bounded on every rearrangement-invariant Banach function
space \(\mathcal B(X)\).

In particular, if \(\{T_r\}_{r>0}\) is a family of operators which is
uniformly bounded on \(L^1(X)\) and \(L^\infty(X)\), then it is
uniformly bounded on \(\mathcal B(X)\).
\end{remark}

\section{The Euclidean Case}\label{sec:euclidean}

We aim to investigate the structure of the space
$\mathcal B_0$.
In this section, we consider the special case where the underlying space is
${\mathbb R}^n$.
By revisiting the proofs of the propositions in Section~\ref{s4},
we observe that the system of Christ dyadic cubes may be replaced by
the standard dyadic cubes
which we define shortly.

Assume that the averaging operators
$\{A_r\}_{r>0}$ are uniformly bounded on $\mathcal B({\mathbb R}^n)$.
Denote by $Q_{jk}$ the dyadic cube
\[
Q_{jk}=2^{-j}\prod_{l=1}^n [k_l,k_l+1),
\]
where $k=(k_1,\dots,k_n)\in{\mathbb Z}^n$.
In this setting, we may choose the partition of unity in the form
\[
\psi_Q(x)=\psi(2^j x-k),
\]
where
$\psi\in C_{\mathrm c}({\mathbb R}^n)$
is a fixed function supported on $[-2,2]^n$.

The \emph{diamond subspace} of $\mathcal B({\mathbb R}^n)$, denoted by
$\overset{\diamond}{\mathcal B}({\mathbb R}^n)$, is defined by
\[
\overset{\diamond}{\mathcal B}({\mathbb R}^n)
=
\left\{
f\in\mathcal B({\mathbb R}^n):
\lim_{t\to0}e^{t\Delta}f=f
\text{ in }\mathcal B({\mathbb R}^n)
\right\}.
\]
Equivalently,
\[
\overset{\diamond}{\mathcal B}({\mathbb R}^n)
=
\left\{
f\in\mathcal B({\mathbb R}^n):
\lim_{t\to0}\psi_t*f=f
\text{ in }\mathcal B({\mathbb R}^n)
\right\},
\]
where
$\psi\in C_{\mathrm c}^{\infty}({\mathbb R}^n)$
satisfies
\[
\int_{{\mathbb R}^n}\psi(x)\,dx=1,
\]
and
\[
\psi_t(x)=t^{-n}\psi(x/t).
\]
See \cite{HNS17,SYYZ26}.
The following theorem identifies the diamond subspace with the subspace
of functions having absolutely continuous norm.
\begin{theorem}\label{lem:diamond=B0}
Assume that $\{A_r\}_{r>0}$ is uniformly bounded on $\mathcal B({\mathbb R}^n)$.
Then
\[
\overset{\diamond}{\mathcal B}({\mathbb R}^n)
=
\mathcal B_0({\mathbb R}^n).
\]
\end{theorem}

\begin{proof}
We first prove that
\[
\mathcal B_0({\mathbb R}^n)\subset
\overset{\diamond}{\mathcal B}({\mathbb R}^n).
\]
Since
$\overset{\diamond}{\mathcal B}({\mathbb R}^n)$
is a closed subspace of $\mathcal B({\mathbb R}^n)$,
it is enough to show that
$T_kf\in\overset{\diamond}{\mathcal B}({\mathbb R}^n)$.

By the definition of $T_kf$,
\[
T_kf=\sum_{m\in\mathbb Z^n}a_m\psi(2^k\cdot-m),
\]
where each $a_m=a_m(f)$ denotes the average of $f$ over $Q_{k m}$.
Hence
\[
\psi_t*(T_kf)-T_kf
=
\sum_{m\in\mathbb Z^n}
a_m
\bigl(\psi_t*\psi-\psi\bigr)(2^k\cdot-m).
\]
Therefore,
using $A_r$ with $r=2^{2+n-k}$, we have
\[
\bigl|\psi_t*(T_kf)-T_kf\bigr|
\lesssim
\|\psi_t*\psi-\psi\|_{L^\infty}
A_{2^{2+n-k}}(|f|).
\]
Since
\[
\|\psi_t*\psi-\psi\|_{L^\infty}\to0
\quad (t\to0),
\]
and
$A_{2^{2+n-k}}$
is bounded on $\mathcal B({\mathbb R}^n)$,
we conclude that
\[
\psi_t*(T_kf)\to T_kf
\quad\text{in }\mathcal B({\mathbb R}^n).
\]
Thus
$T_kf\in\overset{\diamond}{\mathcal B}({\mathbb R}^n)$,
which proves
$
\mathcal B_0({\mathbb R}^n)\subset
\overset{\diamond}{\mathcal B}({\mathbb R}^n).
$

The reverse inclusion is proved by the same approximation argument as in the proof of Theorem~\ref{thm:main}, using the mollified functions $\psi_t*f$ in place of $T_kf$.
\end{proof}

\section{Examples of function spaces}
\label{section:Examples}

\subsection{Morrey spaces}
\label{section:Morrey spaces}

Let \(1\leq p\leq p_0<\infty\).

For a measurable function \(f\) defined on ${\mathbb R}^n$, define
\[
M_{p_0,p}(f;x,r)
:=
|B(x,r)|^{\frac1{p_0}-\frac1p}
\left(
\int_{B(x,r)}|f(y)|^p\,dy
\right)^{\frac1p}.
\]
The Morrey norm of $f$ is defined by
\begin{equation}\label{eq:130709-1A}
\|f\|_{\mathcal M^{p_0}_p}
:=
\sup_{x\in\mathbb R^n,\ R>0}
M_{p_0,p}(f;x,R).
\end{equation}
The Morrey space
\(\mathcal M^{p_0}_p(\mathbb R^n)\)
consists of all functions
\(f\in L^p_{\mathrm{loc}}(\mathbb R^n)\)
such that
$
\|f\|_{\mathcal M^{p_0}_p}<\infty .
$
Notice that
$
\mathcal M^{p_0}_{p_0}(\mathbb R^n)
=
L^{p_0}(\mathbb R^n)
$
with equality of norms.

Compactness properties of Morrey spaces have been studied extensively;
see, for example, Almeida and Samko \cite{AlSa}.
Following this approach, we introduce several vanishing Morrey spaces.

We follow \cite{AlSa} to present some definitions.
The following definition describes the vanishing behavior of the Morrey norm at small scales.
\begin{definition}{\rm \cite{AlSa}}\label{def:V0-Morrey}
The vanishing Morrey space at the origin,
denoted by \(V_0\mathcal M^{p_0}_p(\mathbb R^n)\),
is the subspace of
\(\mathcal M^{p_0}_p(\mathbb R^n)\)
consisting of all functions \(f\) satisfying
\[
\lim_{r\to0}
\sup_{x\in\mathbb R^n}
M_{p_0,p}(f;x,r)=0 .
\]
\end{definition}
In light of Definition~\ref{defi:260722-6},
we may identify
\[
({\mathcal M}^{p_0}_p)_{\mathrm c}({\mathbb R}^n)
=
V_0\mathcal M^{p_0}_p({\mathbb R}^n).
\]

Next, we introduce the corresponding vanishing condition at large scales.
\begin{definition}{\rm \cite{AlSa}}\label{def:V-infinity}
The space
\(V_\infty\mathcal M^{p_0}_p(\mathbb R^n)\)
consists of all functions
\(f\in\mathcal M^{p_0}_p(\mathbb R^n)\)
such that
\[
\lim_{r\to\infty}
\sup_{x\in\mathbb R^n}
M_{p_0,p}(f;x,r)=0 .
\]
\end{definition}

To describe the behavior at infinity after truncation, let
\[
\chi_N=\chi_{\mathbb R^n\setminus B(0,N)}
\]
and define
\[
A_{N,p}(f)
:=
\sup_{x\in\mathbb R^n}
\left(
\int_{B(x,1)}
|f(y)|^p\chi_N(y)\,dy
\right)^{\frac1p}.
\]
The following condition describes 
how the local Morrey norm outside large balls vanishes.
\begin{definition}{\rm \cite{AlSa}}\label{def:V-star}
The space
\(V^{(*)}\mathcal M^{p_0}_p(\mathbb R^n)\)
is defined as the set of all functions
\(f\in\mathcal M^{p_0}_p(\mathbb R^n)\)
such that
\[
\lim_{N\to\infty}A_{N,p}(f)=0 .
\]
\end{definition}
In light of Definition \ref{defi:260722-5},
we can say that
\(V^{(*)}\mathcal M^{p_0}_p(\mathbb R^n)
=(\mathcal M^{p_0}_p)_\infty(\mathbb R^n).\)

Combining the above three vanishing conditions, we define the following subspace of Morrey spaces.
\begin{definition}{\rm \cite{AlSa}}\label{def:V-star-0-infinity}
We define
\[
V^{(*)}_{0,\infty}
\mathcal M^{p_0}_p(\mathbb R^n)
:=
V_0\mathcal M^{p_0}_p(\mathbb R^n)
\cap
V_\infty\mathcal M^{p_0}_p(\mathbb R^n)
\cap
V^{(*)}\mathcal M^{p_0}_p(\mathbb R^n).
\]
\end{definition}

Using Theorems \ref{thm:characterization} and
\ref{lem:diamond=B0},
we characterize the closure of compactly supported functions in Morrey spaces.
\begin{theorem}[Characterization of the separable Morrey subspace]
\label{thm:Morrey-characterization}
Let \(1\leq p\leq p_0<\infty\).
Assume that the dyadic averaging operators are uniformly bounded on
\(\mathcal M^{p_0}_p({\mathbb R}^n)\).
Then
\begin{align*}
\widetilde{\mathcal M}^{p_0}_p({\mathbb R}^n)
=
\overline{C_{\rm c}({\mathbb R}^n)}^{\,\mathcal M^{p_0}_p({\mathbb R}^n)}
&=
\left\{
f\in\mathcal M^{p_0}_p({\mathbb R}^n):
\begin{aligned}
&
\lim_{r \to 0}
\|f-A_rf\|_{{\mathcal M}^{p_0}_p}=0,
\\
&
\lim_{r \to \infty}
\|f\chi_{B(0,r)^{\rm c}}\|_{{\mathcal M}^{p_0}_p}=0
\end{aligned}
\right\}\\
&
=
\left\{
f\in\overset{\diamond}{\mathcal M}{}^{p_0}_p({\mathbb R}^n):
\lim_{r \to \infty}
\|f\chi_{B(0,r)^{\rm c}}\|_{{\mathcal M}^{p_0}_p}=0
\right\}.
\end{align*}
\end{theorem}

We next establish a compactness criterion for bounded subsets of Morrey spaces
as a corollary of Theorem \ref{thm:main}.
\begin{theorem}[Compactness criterion in Morrey spaces]
\label{thm:compact-Morrey}
Let
\[
\mathcal F\subset
V^{(*)}_{0,\infty}
\mathcal M^{p_0}_p(\mathbb R^n).
\]
Assume that

\[
\sup_{f\in\mathcal F}
\|f\|_{\mathcal M^{p_0}_p}<\infty
\]
and that the averaging operators satisfy
\[
\lim_{r\to0}
\sup_{f\in\mathcal F}
\|f-A_rf\|_{\mathcal M^{p_0}_p}=0 .
\]
Then \(\mathcal F\) is relatively compact in
\(\mathcal M^{p_0}_p(\mathbb R^n)\).
\end{theorem}
See also \cite{HaSa25} for another characterization
of compact sets in $\widetilde{{\mathcal M}}{}^{p_0}_p({\mathbb R}^n)$.

We conclude this section with an example showing that boundedness in the Morrey norm does not imply the required vanishing property.
\begin{example}
Let $1 \le p<p_0<\infty$.
Then
$f(x)=|x|^{-\frac{n}{p_0}} \in {\mathcal M}^{p_0}_p({\mathbb R}^n)$
\cite{SaFaHa}.
However,\[
\lim_{R\to\infty}
\|f\chi_{B(0,R)}\|_{\mathcal M^{p_0}_{p}}
=0.
\]
fails.
\end{example}
\subsection{$L^\infty({\mathbb R})$}

This example shows that we can not develop
our theory for $L^\infty({\mathbb R})$.

\begin{example}
Let $\mathcal B(\mathbb R)=L^\infty(\mathbb R)$, and define
\[
\mathcal B_{\rm loc}(\mathbb R)
=
\left\{
f\in L^\infty(\mathbb R):
\lim_{j\to\infty}
\|f\chi_{E_j}\|_{L^\infty(\mathbb R)}
=0
\text{ whenever }E_j\downarrow\varnothing
\right\}.
\]
Then
\[
C_{\mathrm c}(\mathbb R)\not\subset\mathcal B_{\rm loc}(\mathbb R).
\]
Consider the function
\[
f(x)
=
\max\{1-|x|,0\},
\qquad x\in\mathbb R.
\]
Clearly,
$f\in C_{\mathrm c}(\mathbb R)$.

For each $n\in\mathbb N$, let
\[
E_n
=
\left[-\frac1n,\frac1n\right].
\]
Then
\[
E_1\supset E_2\supset\cdots,
\qquad
\bigcap_{n=1}^\infty E_n=\{0\}.
\]

Since a singleton has Lebesgue measure zero, we have
$E_n\downarrow\varnothing$ modulo null sets.
Moreover,
\[
\|f\chi_{E_n}\|_{L^\infty(\mathbb R)}
=
\sup_{x\in E_n}f(x)
=
f(0)
=
1
\]
for every $n\in\mathbb N$.
Hence
\[
\lim_{n\to\infty}
\|f\chi_{E_n}\|_{L^\infty(\mathbb R)}
=
1\neq0.
\]
Therefore,
$f\notin\mathcal B_{\rm loc}(\mathbb R)$, which proves that
\[
C_{\mathrm c}(\mathbb R)\not\subset\mathcal B_{\rm loc}(\mathbb R).
\]
\end{example}

\section*{Data availability}

Not applicable.

\section*{Author contributions}

The four authors contributed equally to the
correctness of this paper.

\end{document}